\documentclass[10pt]{amsart}

\usepackage[colorlinks]{hyperref}
\usepackage{color,graphicx,shortvrb}
\usepackage[latin 1]{inputenc}
\usepackage{multicol}

\usepackage[active]{srcltx} 

\usepackage{enumerate}

\newtheorem{theorem}{Theorem}[section]

\newtheorem{corollary}[theorem]{Corollary}

\newtheorem{lemma}[theorem]{Lemma}

\newtheorem{problem}[theorem]{Problem}
\newtheorem{proposition}[theorem]{Proposition}
\newtheorem{remark}[theorem]{Remark}

\newtheorem{example}[theorem]{Example}

\def\J#1#2#3{ \left\{ #1,#2,#3 \right\} }
\def\RR{{\mathbb{R}}}

\def\11{\textbf{$1$}}
\def\CC{{\mathbb{C}}}
\def\HH{{\mathbb{H}}}

\DeclareGraphicsExtensions{.jpg,.pdf,.png,.eps}

\usepackage{enumerate}

\begin{document}

\title[Real Local triple derivations]{Local triple derivations on real C$^*$-algebras and JB$^*$-triples}
\date{}

\author[Fern\'{a}ndez-Polo]{Francisco J. Fern\'{a}ndez-Polo}
\email{pacopolo@ugr.es}
\address{Departamento de An{\'a}lisis Matem{\'a}tico, Facultad de
Ciencias, Universidad de Granada, 18071 Granada, Spain.}

\author[Molino]{Alexis Molino}
\email{alexis.molino@gmail.com}
\address{Departamento de An{\'a}lisis Matem{\'a}tico, Facultad de
Ciencias, Universidad de Granada, 18071 Granada, Spain.}

\author[A.M. Peralta]{Antonio M. Peralta}
\email{aperalta@ugr.es}
\address{Departamento de An{\'a}lisis Matem{\'a}tico, Facultad de
Ciencias, Universidad de Granada, 18071 Granada, Spain.}

\thanks{Authors partially supported by the Spanish Ministry of Economy and Competitiveness,
D.G.I. project no. MTM2011-23843, and Junta de Andaluc\'{\i}a grant FQM3737.}



\date{February, 2013}

\maketitle
 \thispagestyle{empty}

\begin{abstract} We study when a local triple derivation on a real JB$^*$-triple is a triple derivation. We find an example of a (real linear) local triple derivation on a rank-one Cartan factor of type $I$ which is not a triple derivation. On the other hand, we find sufficient conditions on a real JB$^*$-triple $E$ to guarantee that every local triple derivation on $E$ is a triple derivation.
\end{abstract}

\section{Introduction}

Derivations on Banach algebras, C$^*$-algebras, and more recently on real and complex JB$^*$-triples, constitute a field of interest for the researcher in Functional Analysis and operator algebras. Derivations on C$^*$-algebras constitute one of the most studied subclasses of the class of linear operators and they have been considered since the very earliest stages of development of the theory of Banach algebras.  One of the reasons for this interest is that each $^*$-derivation on a C$^*$-algebra $A$ (respectively, each triple derivation on a JB$^*$-triple $E$) generates a one parameter group of $^*$-automorphisms on $A$ (respectively, a one parameter group of $^*$-triple-automorphisms on $E$). We recall that an (associative) \emph{derivation} on a Banach algebra $A$ is a linear mapping $D: A \to A$ satisfying $D(a b) = D(a) b + a D(b),$ for every $a,b$ in $A$. It is due to S. Sakai \cite{Sak60} that every (associative)
derivation on a C$^*$-algebra is automatically continuous; actually,  J.R. Ringrose proved that the same statement holds for every (associative)
derivation from a C$^*$-algebra $A$ to a Banach $A$-bimodule $M$ (compare \cite{Ringrose72}).\smallskip

In 1990, R.V. Kadison introduced the concept of local derivation. A linear mapping $T$ on a Banach algebra $A$ is a \emph{local derivation} if for each $a$ in $A$ there is a derivation $D_a$ on $A$ with $D_a (a) = T(a).$ Kadison's contribution on local derivations determines that every bounded local derivation on a von Neumann algebra (i.e. a C$^*$-algebra which is also a dual Banach space) is a derivation (cf. \cite{Kad90}). Concerning local derivations on C$^*$-algebras, the culminating result is due to B.E. Johnson \cite{John01}, who proved that every bounded local derivation from a C$^*$-algebra $A$ into a Banach $A$-bimodule is a derivation. Furthermore, local derivations on C$^*$-algebras are continuous even if not assumed
a priori to be so.\smallskip

In recent years, the research on local derivation gave rise to the study of local triple derivations on the wider class of JB$^*$-triples. A simple motivation relies in the fact that every C$^*$-algebra, when equipped with the ternary product $ \J xyz =\frac12 (x y^* z +z y^* x),$ lies in the wider category of JB$^*$-triples. A JB$^*$-triple is essentially a complex Banach space $E$ equipped with a triple product $\J ... : E\times E\times E \to E$ which is conjugate linear in the middle variable and symmetric and linear in the outer variables satisfying certain algebraic-analytic axioms (see Subsection 1.1 for more details). A \emph{triple derivation} on a JB$^*$-triple $E$ is a linear mapping $\delta: E\to E$ satisfying Leibnitz's rule for the triple product, that is, \begin{equation}
\label{eq Leibnitz triple} \delta \J abc = \J {\delta(a)}bc + \J a{\delta(b)}c + \J ab{\delta(c)},
\end{equation} for every $a,b,c\in E$. Inspired by the research line opened by Ringrose and Sakai, T. Barton and Y. Friedman proved that triple derivations on JB$^*$-triples are automatically continuous (cf.  \cite{BarFri}).\smallskip

Local theory for triple derivations was initiated by M. Mackey in \cite{Mack}. A \emph{local triple derivation} on a JB$^*$-triple $E$ is a linear map $T : E\to E$ such that for each $a$ in $E$ there exists a triple derivation $\delta_{a}$ on $E$ satisfying $T(a) = \delta_a (a).$ In the just mentioned paper, Mackey proves an analogous result to that of Kadison for local triple derivations, and shows that every bounded local triple derivation on a JBW$^*$-triple (i.e. a JB$^*$-triple which is also a dual Banach space) is a triple derivation. A similar result for bounded local triple derivations on unital C$^*$-algebras was established by M. Burgos, J.J. Garcés and the first and third authors of this note (see \cite{BurFerGarPe2012}). The problems and questions about triple derivations have been completely solved in the recent note \cite{BurFerPe2013}, where M. Burgos and the first and third authors of this note establish that
every continuous local triple derivation on a JB$^*$-triple is a triple derivation; moreover, local triple derivations on a JB$^*$-triple are continuous even if not assumed a priori to be so. In the same paper, the authors open the scope to the study of real-linear local triple derivations on JB$^*$-triples and to the study of local triple derivations in the wider class of real JB$^*$-triples (i.e. norm-closed real subtriples of JB$^*$-triples). In the real setting it is proved that every local triple derivation $T$ on a real JB$^*$-triple $E$ is continuous and for every $x$ in $E$ we have: $$ T \J xxx = 2 \J {T(x)}xx + \J x{T(x)}x.$$ When $E$ is a (complex) JB$^*$-triple and $T$ is complex linear, the polarisation formula $$ 8\{x,y,x\}=\sum_{k=0}^3 \sum_{j=1}^2 i^k(-1)^j \J {x+i^ky+(-1)^jz}{x+i^ky+(-1)^jz}{x+i^ky+(-1)^jz}$$ gives $$ T \J xyz =  \J {T(x)}yz + \J x{T(y)}z + \J xy{T(z)},$$ and hence $T$ is a triple derivation. In the real setting, the lacking of a polarisation formula in the terms given above, makes invalid the above argument. Similar difficulties appear in the studies of (surjective) real linear isometries between JB$^*$-triples and real JB$^*$-triples developed by Ch.-H. Chu, T. Dang, B. Russo, and B. Ventura \cite{ChuDaRuVen}, T. Dang \cite{Da}, J.M. Isidro, W. Kaup and A. Rodríguez \cite{IsKaRo95}, F.J. Fern\'andez-Polo, J. Mart\'inez and A.M. Peralta \cite{FerMarPe} and M. Apazoglou and A.M. Peralta \cite{ApazPe}. The current knowledge concerning local triple derivations on real JB$^*$-triples leads us to the following problem:

\begin{problem}\label{problem local triple derivations on real JB*-triples}{\rm \cite[Problem 2.6]{BurFerPe2013}} Is every (bounded) local triple derivation on a real JB$^*$-triple a triple derivation?
\end{problem}

In this paper we give a complete answer to the above question. Unfortunately, the real setting provides some difficulties and obstacles, and the answer to Problem \ref{problem local triple derivations on real JB*-triples} is, in general, negative (compare Example \ref{example local not derivation}). However, in our main result (Theorem \ref{t sufficient conds on general real JB*-triples}) we shall determine sufficient conditions on a real JB$^*$-triple $E$ to guarantee that every local triple derivation on $E$ is a triple derivation. Among the consequences it follows that every local triple derivation on a real C$^*$-algebra is a triple derivation. It is superfluous to say that the techniques needed in the real setting are completely different and independent from those employed in the case of complex JB$^*$-triples.

\subsection{Notation and preliminaries}

A complex Banach space, ${E}$, together with a triple product $\J ... :{E} \times {E}
\times {E} \rightarrow {E}$, which is continuous, symmetric
and linear in the outer variables and conjugate linear in the
inner one, satisfying:
\begin{enumerate}[$a)$]
\item \textit{Jordan Identity}:
$$
L(a,b) \J xyz = \J {L(a,b)x}yz - \J x{L(b,a)y}z + \J xy{L(a,b)z},
$$ for all $a,b,x,y,z\in {E}$, where $L(a,b)x:= \J abx$;
\item For each $a\in {E},$ the operator $L(a,a)$ is hermitian with non-negative spectrum
and $\Vert \J aaa \Vert =\Vert a\Vert ^{3}$,
\end{enumerate} is said to be a (complex) \emph{JB$^*$-triple}.\smallskip

Every C$^*$-algebra is a complex JB$^*$-triple with respect to the
triple product \linebreak $\J xyz = \frac{1}{2} ( x y^* z + z y^*
x)$, and in the same way every JB$^*$-algebra with respect to $\J
abc = \left( a\circ b^{*}\right) \circ c+\left( c\circ
b^{*}\right) \circ a-\left( a\circ c\right) \circ b^{*}$.
Other examples of JB$^*$-triples are given by the so-called building blocks of JB$^*$-triples. We refer to the (complex) \emph{Cartan factors} of type $1$ to $6$ defined as follows: the Cartan factor of type 1 (also denoted by $I^{\mathbb{C}}$) is the Banach space $L(H, K)$ of bounded linear operators between two complex Hilbert spaces, $H$ and $K$, where the triple product is defined by $\J xyz= 2^{-1}(xy^*z+zy^*x)$. Cartan factors of types 2 and 3 are the subtriples of $L(H)$ defined by $II^{\mathbb{C}} = \{ x\in L(H) : x=- j x^* j\} $ and $III^{\mathbb{C}} = \{ x\in L(H) : x= j x^* j\}$, respectively, where $j$ is a conjugation on $H$. A Cartan factor of type 4 or $IV$ is a complex spin factor, that is, a complex Hilbert space provided with
a conjugation $x \mapsto \overline{x}$, triple product $$\J x y z
= \left< x / y \right> z + \left< z / y \right> x - \left< x /
\bar z \right> \bar y,$$ and norm given by $\| x\|^2=\left< x / x
\right>+\sqrt {\left< x / x \right>^2-|\left< x / \overline x
\right>|^2}$. The Cartan factors of types 5 and 6 consist of matrices over the eight
dimensional complex Cayley division algebra $\mathbb{O}$; the type $VI$ is the space of all hermitian $3$x$3$ matrices over $\mathbb{O}$, while the type $V$ is the subtriple of $1$x$2$ matrices with entries in $\mathbb{O}$ (compare \cite{Loos77}, \cite{FriRu2}, \cite{DanFri87} and \cite{Ka97}).\smallskip

In 1995, J.M. Isidro, W. Kaup and A. Rodríguez introduce a class of real Banach spaces, called real JB$^*$-triples, containing all real and complex C$^*$-algebras and all (complex) JB$^*$-triples. A \emph{real JB$^*$-triple} is a norm-closed real subtriple of a (complex) JB$^*$-triple (cf. \cite{IsKaRo95}). Every real JB$^*$-triple $E$ can be also regarded as a \emph{real form} of a complex JB$^*$-triple, that is, there exist a (complex) JB$^*$-triple $E_{c}$ and a conjugate linear isometry $\tau: E_c\rightarrow E_c$ of period 2
such that $E=\{b\in E_c\::\:\tau(b)=b\}$. We can actually identify $E_{c}$ with the
complexification of $E$.\smallskip

Real forms of Cartan factors are called \emph{real Cartan factors}. W. Kaup classified all real Cartan factors in
\cite[Corollary 4.4]{Ka97} (the classification in the finite dimensional case is due to O. Loos \cite[pages 11.5-11.7]{Loos77}). Real Cartan factors as classified, up to triple isomorphisms, as follows: Let $X$ and $Y$ be two real Hilbert spaces, let $P$ and $Q$ be two Hilbert spaces over the quaternion field
$\mathbb{H}$, and finally, let $H$ be a complex Hilbert.

\begin{enumerate}
\begin{multicols}{2}

\item $I^{\RR} := \mathcal{L}(X,Y)$\vspace{0.45cm}

\item $I^{\HH} := \mathcal{L}(P,Q)$\vspace{0.45cm}

\item $II^{\CC}:=\{z\in
\mathcal{L}(H):z^*=z\}$\vspace{0.45cm}

\item $II^{\RR}:=\{x\in
\mathcal{L}(X):x^*=-x\}$\vspace{0.45cm}

\item $II^{\HH}:= \{w\in
\mathcal{L}(P):w^*=w\}$\vspace{0.45cm}

\item $III^{\RR} :=\{x\in \mathcal{L}(X): x^* =
x\}$\vspace{0.45cm}

\item $III^{\HH}\!:=\!\{w\in\!\mathcal{L}(P)\!:\! w^*\!
=\!-w\}$\vspace{0.45cm}
\end{multicols}

\item $IV^{r,s}:=E,$ where  $E=X_1 \oplus^{\ell_{_1}}
X_2$ and $X_1$,$X_2$ are closed linear subspaces, of dimensions
$r$ and $s$, of a real Hilbert space, $X$, of dimension greater or
equal to three, so that $X_2={X_1}^{\perp}$, with triple product
$$\J x y z = \left< x / y \right> z + \left< z / y \right> x -
\left< x / \bar z \right> \bar y,$$ where $\left< . / . \right>$
is the inner product in $X$ and the involution $x\to \bar x$ on
$E$ is defined by $\bar x = ( x_1,-x_2)$ for every $x=(x_1,x_2)$.
This factor is known as a \emph{real spin factor}.

\begin{multicols}{2}

\item $V^{\mathbb{O}_{0}}:=M_{1,2}(\mathbb{O}_{0})$\vspace{0.45cm}
\item $V^{\mathbb{O}}:= M_{1,2} (\mathbb{O})$\vspace{0.45cm} \item
$VI^{\mathbb{O}_{0}}:=H_3 (\mathbb{O}_{0})$\vspace{0.45cm} \item
$VI^{\mathbb{O}}:=H_3 (\mathbb{O})$\vspace{0.45cm}
\end{multicols}

\noindent where $\mathbb{O}_{0}$ is the real split Cayley algebra
over the field of the real numbers and $\mathbb{O}$ is the real
division Cayley algebra (known also as the algebra of real
division octonions). The real Cartan factors $(ix)-(xii)$ are
called \emph{exceptional real Cartan factors}.\end{enumerate}

By a \emph{generalized real Cartan factor} we shall mean a real Cartan factor or a complex Cartan factor regarded as a real JB$^*$-triple.\smallskip

A real or complex JBW$^*$-triple is a JB$^*$-triple which is also a dual Banach space. The second dual of a real or complex JB$^*$-triple is a JBW$^*$-triple (see \cite{Di86b}, \cite{IsKaRo95}). Every real or complex JBW$^*$-triple admits a unique (isometric) predual and its product is separately weak$^*$-continuous (compare \cite{BarTi} and \cite{MarPe}).\smallskip

For each element $a$ in a real or complex JB$^*$-triple, $E$, the symbol $Q(a)$ will denote the mapping on $E$ defined by $Q(a) (x) = \J axa.$\smallskip

An element $e$ in a real or complex JB$^*$-triple, $E,$ is called a tripotent whenever $\J eee = e$. Every tripotent $e \in E$ induces a decomposition of $E,$ $$ E= E_{0} (e) \oplus E_{1} (e) \oplus
E_{2} (e),$$ where
$E_{k} (e) := \{ x\in E : L(e,e)x = \frac{k}{2} x \}$ is a
subtriple of $E$ (compare \cite[Theorem 3.13]{Loos77}). The natural projection of $E$ onto
$E_{k} (e)$ will be denoted by $P_{k} (e)$. This decomposition is called the
Peirce decomposition associated with the tripotent $e$. The
following \emph{Peirce arithmetic} is satisfied: $$\J {E_0 (e)}{E_2 (e)}{E} = \J {E_2 (e)}{E_0 (e)}{E } = 0,$$ and
$$\J {E_{i}(e)}{E_{j}(e)}{E_{k}(e)} \subseteq E_{i-j+k}(e),$$ where $E_{l}(e)=0$ for every $l\neq 0,1,2.$ A tripotent $e$ in $E$ is called
\emph{minimal} when $E^{1} (e):=\{ x\in E : Q(e) (x) =x  \} = \RR e$.\smallskip

Two elements $a,b$ in a real or complex JB$^*$-triple $E$ are said to be
\emph{orthogonal} (written $a\perp b$) if $L(a,b) =0$. It is known that $a\perp b$ $\Leftrightarrow \J aab=0$ $\Leftrightarrow \J bba=0$ (cf. \cite[Lemma 1]{BurFerGarMarPe}). We shall say that two sets $A$,$B\subseteq E$ are orthogonal 
if $a\perp b$, for every $a\in A$ and $b\in B$. 
\smallskip

The rank, $r(E)$, of a real or complex JB$^*$-triple $E$, is the minimal cardinal number $r$
satisfying $card(S)\leq r$ whenever $S$ is an orthogonal subset of
$E$, i.e. $0\notin S$ and $x\perp y$ for every $x\neq y$ in $S$.\smallskip

Let $E$ be a real JB$^*$-triple. It is known that the JB$^*$-subtriple $E_a$ generated by a single element $a$ in $E$ is JB$^*$-triple isomorphic (and hence isometric) to $C_0(L,\mathbb{R})$ for some locally compact Hausdorff space $L\subseteq (0,\|a\|],$ such that $L\cup \{0\}$ is compact. It is also known the existence of a triple isomorphism $\Psi$ from $E_a$ onto $C_{0}(L,\mathbb{R}),$ satisfying $\Psi (a) (t) = t$ $(t\in L)$ (compare
\cite[Lemma 1.14]{Ka} and \cite[\S 3, page 81]{BurPeRaRu}). It follows that in a real JB$^*$-triple of rank one every norm-one element is a tripotent.\smallskip

Given two elements $a,b$ in a real or complex JB$^*$-triple, $E$, it follows from the
Jordan identity that $\delta (a,b) := L(a,b)-L(b,a)$ is a
triple derivation on $E$. An {\it inner triple derivation} on $E$ is a triple
derivation $\delta$ which can be written as finite sum of derivations of the form
$\delta(a,b)$ with $a,b\in E$, i.e.,
$$ \delta = \sum_{j=1}^{n} \delta(a_{j},b_{j}).$$ We say that $E$ satisfies the {\it inner
derivation property} (IDP) when every triple derivation on $E$ is
inner. Unfortunately, there exist examples of real and complex JB$^*$-triples which do not satisfy the IDP (compare \cite{HoMarPeRu}). When the space of all inner triple derivations on $E$ is dense in the space of all triple derivations on $E$, with respect to the strong operator topology of $L(E)$, we shall say that $E$ has the \emph{inner approximation property} (IAP for short). The advantage of the latter property being that every real or complex JB$^*$-triple has the IAP (see \cite[Theorem 4.6]{BarFri} and \cite[Theorem 5]{HoMarPeRu}).

\section{Properties of local triple derivations on real Cartan factors}

In this section we shall determine those generalized real Cartan factors $C$ such that every local triple derivation on $C$ is a triple derivation.\smallskip

Let $T: E\to E$ be a local triple derivation on a real JB$^*$-triple. Corollary 2.5 in \cite{BurFerPe2013} assures that $T\J aaa = 2\J {T(a)}aa + \J a{T(a)}a$, for every $a\in E.$ If we consider the symmetrized triple product $ <\!a,b,c\!>:=\frac13 \left(\J abc + \J cab + \J bca\right),$ which is trilinear and symmetric, a real polarisation formula 
gives that $T$ is a triple derivation of the symmetrized Jordan triple product $ <\!.,.,.\!>.$ In other words, \cite[Corollary 2.5]{BurFerPe2013} shows that every local triple derivation on a real JB$^*$-triple is a derivation of the symmetrized triple product. We shall see in this section that the reciprocal statement is true whenever $E$ is a generalized real Cartan factor. In a first step we consider real JB$^*$-triples of rank one.\smallskip

We recall now the behavior of a local triple derivation on a tripotent element. Let $T: E\to E$ be a local triple derivation on a real JB$^*$-triple and let $e$ be a tripotent in $E.$  Then, by \cite[$(6)$, page 6]{BurFerPe2013},
\begin{equation}\label{eq local triple derivations on tripotents} P_0 (e) (T(e))=0 \hbox{ and } P_2 (e) (T(e)) = - Q (e) (T(e)).
\end{equation}

\begin{proposition}\label{p local derivations in rank-one cartan factors}
Let $E$ be a real JB$^*$-triple of rank one. Every (real linear) triple derivation of the symmetrized Jordan triple
product, $T : E \to E$, is a local triple derivation.
\end{proposition}

\begin{proof} We have to show that, for every $x$ in $E$, there exists a triple derivation $\delta:E \to E$ such that $T(x)=\delta (x)$. To this end, let $x$ be a non-zero element in $E$ and set $u:=\frac{x}{\|x\|}$. We notice that in a real JB$^*$-triple of rank one every norm-one element is a minimal tripotent (compare subsection 1.1). Finally, it is easily verified, via Peirce arithmetic and \eqref{eq local triple derivations on tripotents}, that the (inner) triple derivation $$\delta = \frac{1}{2 \|x\|} \delta(T(x)+3P_1(u)T(x),u),$$ 
satisfies $$\delta(x) = \frac12 \Big( \J {T(x)}uu + 3 \J { P_1 (u)T(x)}uu - \J u{T(x)}u - 3 \J u{P_1 (u)T(x)}u\Big)$$ $$=\frac12 \|x\| \Big( \J {T(u)}uu + 3 \J { P_1 (u)T(u)}uu - \J u{T(u)}u  \Big)$$ $$= P_2 (u)T(x) + P_1 (u)T(x) =T(x).$$
\end{proof}

Haplessly, we do not know if the statement in the above Proposition \ref{p local derivations in rank-one cartan factors} remains valid for every real JB$^*$-triple. That is, we do not know whether every triple derivation of the symmetrized triple product on a real JB$^*$-triple is a local triple derivation. However, Propositions \ref{p local derivations in rank-one cartan factors} and \ref{p Cartan factors of rank bigger than one} guarantee that local triple derivations and triple derivation of the symmetrized triple product on a generalized real Cartan factor are the same notions.\smallskip

We shall exhibit, in the sequel, a series of results guaranteing that in some particular generalized real Cartan factor of rank one (concretely, real Hilbert spaces regarded as real Cartan factor of type $I^{\mathbb{R}}$ and real spin factors), every local triple derivation is a triple derivation.

\begin{lemma}\label{l realtypeI} Let $E$ be a generalized real Cartan factor of type $I^{\mathbb{R}}$, and let  $T: E \rightarrow E$ be a real linear mapping. The following statements are equivalent:
   \begin{enumerate}[$(a)$] \item $T$ is a local triple derivation;
\item $T$ is a triple derivation for the symmetrized triple product;
\item $T$ is a bounded skew-symmetric operator (i.e. $T^*=-T$);
\item $T$ is a triple derivation.
\end{enumerate}
\end{lemma}

\begin{proof} We recall that $E$ is a real Hilbert space and the triple product of $E$ is given by  $$\{x,y,z\}=\frac{1}{2}(<x/y>z+<z/y>x),$$ where $<\cdot/\cdot>$ is the inner product on $E$.\smallskip

The implication $(a)\Rightarrow (b)$ is proved in \cite[Corollary 2.5]{BurFerPe2013}. The equivalence $(c)\Leftrightarrow (d)$ was established in \cite[Lemma 3, Section 3.3]{HoMarPeRu}, while $(d)\Rightarrow (a)$ is obvious.\smallskip

We shall prove $(b)\Rightarrow (c)$. Let $T$ be a derivation for the symmetrized triple product $<\cdot,\cdot,\cdot >$. For each $x \in  E$, we have that  $T\{x,x,x\}=2\{Tx,x,x\}+\{x,Tx,x\}$ and hence $$\Vert x \Vert ^2 T(x) =<T(x)/x> x+\Vert x \Vert ^2 T(x)+ <x/T(x)> x,$$ which gives, $$
     <T(x)/x>=-<x/T(x)>, \hspace{0.6cm} \forall x \in E,$$ and hence $T^* = -T.$
\end{proof}

We deal now with rank-one real spin factors.

\begin{lemma}\label{l caseofspinfactor} Let $E$ be a real spin factor of rank one and let $T: E \rightarrow E$ be a (real) linear mapping. The following statements are equivalent:
\begin{enumerate}[$(a)$] \item $T$ is a local triple derivation;
\item $T$ is a triple derivation for the symmetrized triple product;
\item $T$ is a bounded skew-symmetric operator ($T^*=-T$);
\item $T$ is a triple derivation.
\end{enumerate}
\end{lemma}

\begin{proof}
We recall that $E$ is a real Hilbert space with inner product $<\cdot/\cdot>$, whose triple product is defined by $\{x,y,z\}=<x/y>z+<z/y>x-<x/z>y$ (compare \cite[Theorem 4.1 and Proposition  5.4]{Ka97}).\smallskip

The arguments given in the proof of Lemma \ref{l realtypeI} remain valid to prove the implications $(a)\Rightarrow (b)$, $(d)\Rightarrow (a)$, while $(c)\Leftrightarrow (d)$ was essentially obtained in \cite[Section 3.2]{HoMarPeRu}.\smallskip

$(b)\Rightarrow (c)$ If $T$ is a triple derivation for the symmetrized triple product $<\cdot,\cdot,\cdot >$, then $$T\{x,x,x \}=2 \{T(x),x,x \}+\{x,T(x),x\},$$ $$\Vert x \Vert ^2 T(x)=2<T(x)/x>x+2 \Vert x \Vert ^2 T(x) -2<T(x)/x>x$$ $$ +2<x/T(x)>x-\Vert x \Vert ^2 T(x),$$ and hence $ <x/T(x)>=0,$ for all $x \in E$, which concludes the proof.
\end{proof}

Generalized real Cartan factors of rank one treated in the above Lemmas \ref{l caseofspinfactor} and \ref{l realtypeI} are essentially real Hilbert spaces and every local triple derivation on them is a triple derivation. Unfortunately, there exist other examples of rank-one real JB$^*$-triples having an essentially ``complex'' or ``quaternionic'' nature. The next example shows the existence of rank-one Cartan factors $C$ having a complex structure such that there exists a local triple derivation on $C$ which is not a triple derivation.

\begin{example}\label{example local not derivation} Let $H= \mathbb{C}^2$ be the 2-dimensional complex Hilbert space equipped with its natural inner product $<(\lambda_1,\lambda_2) | (\mu_1,\mu_2)> = \lambda_1 \overline{\mu_1} + \lambda_2 \overline{\mu_2}$. We equip $H$ with its structure of (rank-one) complex Cartan factor of type $I^{\mathbb{C}}$, that is, $$2 \{ \lambda, \mu, \nu \} = <\lambda| \mu> \nu  + <\nu | \mu> \lambda.$$ Of course, we consider the real JB$^*$-triple obtained from $H$ by restricting scalar multiplication to the real numbers.\smallskip

Let $T: H = \mathbb{R}^{4} \to H = \mathbb{R}^{4}$ be the real linear mapping whose associate matrix, in the canonical basis, is given by $\displaystyle\left(
                         \begin{array}{cccc}
                           0 & 0 & 1 & 0 \\
                           0 & 0 & 0 & 0 \\
                           -1 & 0 & 0 & 0 \\
                           0 & 0 & 0 & 0 \\
                         \end{array}
                       \right)$, equivalently, $$T(\lambda_1,\lambda_2)= \left( \Re\hbox{e} (\lambda_2), - \Re\hbox{e} (\lambda_1) \right).$$ Clearly, $T$ is only $\mathbb{R}$-linear. It is not hard to check that $\Re\hbox{e} < T(\lambda_1,\lambda_2)| (\lambda_1,\lambda_2)>=0,$ and thus,
                       $$2\{T(\lambda_1,\lambda_2), (\lambda_1,\lambda_2), (\lambda_1,\lambda_2) \} + \{(\lambda_1,\lambda_2), T(\lambda_1,\lambda_2), (\lambda_1,\lambda_2) \} $$ $$ =< T(\lambda_1,\lambda_2)| (\lambda_1,\lambda_2)> (\lambda_1,\lambda_2) + (|\lambda_1|^2 +|\lambda_2|^2) T(\lambda_1,\lambda_2)$$ $$ + < (\lambda_1,\lambda_2)| T(\lambda_1,\lambda_2)> (\lambda_1,\lambda_2) $$
$$ =  (|\lambda_1|^2 +|\lambda_2|^2) T(\lambda_1,\lambda_2) + 2 \Re\hbox{e} < T(\lambda_1,\lambda_2)| (\lambda_1,\lambda_2)> (\lambda_1,\lambda_2)$$
$$ = (|\lambda_1|^2 +|\lambda_2|^2) T(\lambda_1,\lambda_2) = T\{(\lambda_1,\lambda_2), (\lambda_1,\lambda_2), (\lambda_1,\lambda_2) \},$$ which shows that $T\J xxx = 2 \J {T(x)}xx + \J x{T(x)}x$, for every $x\in H$. A priori, this is not enough to guarantee that $T$ is a local derivation. However, Proposition \ref{p local derivations in rank-one cartan factors} assures that $T$ is a local triple derivation.\smallskip

On the other hand, since  $$T\{(1,0),(i,0),(1,0) \}=(0,0)$$ and $$ 2\{T(1,0),(i,0),(1,0)\} + \{(1,0),T(i,0),(1,0) \} = (0,i),$$
we deduce that $T$ is not a triple derivation.
\end{example}

Motivated by an example given by R.V. Kadison in \cite{Kad90}, M. Mackey provides an example of a local triple derivation on the $^*$-algebra $\mathbb{C}(x),$ of rational functions in the real variable $x$ over $\mathbb{C}$, which is not a triple derivation. The particularity of the real setting provides, in Example \ref{example local not derivation}, a new example of a real linear local triple derivation on a finite dimensional (complex) JB$^*$-triple which is not a triple derivation.\smallskip

We shall give a shorter argument to justify the statements in the above Example \ref{example local not derivation}.
The following observation will be very useful.

\begin{proposition}\label{p derivations are complex linear}
Let $E$ be a complex JB$^*$-triple. Every real linear triple derivation $\delta : E \to E$ is complex linear.
\end{proposition}

\begin{proof} Having in mind that, for each $a,b$ in $E$, $L(a,b): E \to E$ is
$\CC-$linear, every (real linear) inner derivation on $E$ is $\CC-$linear.\smallskip

Suppose now that $\delta: E\to E$ is a real linear derivation.  Since every real JB$^*$-triple satisfies IAP, given $\varepsilon > 0$ and $x\in E$, there exists a inner derivation $ \widehat{\delta} $ such that $\Vert \delta(x)-\widehat{\delta}(x) \Vert< \frac{\varepsilon}{2} $ and $\Vert \delta(ix)-\widehat{\delta}(ix) \Vert<
   \frac{\varepsilon}{2} $. Therefore $$\Vert i\delta(x)-\delta(ix) \Vert \leq \Vert i\delta(x)-i\widehat{\delta}(x) \Vert + \Vert \widehat{\delta}(ix)-\delta(ix) \Vert < \varepsilon.$$ The arbitrariness of $\varepsilon$ and $x$ guarantee the desired statement.
\end{proof}

Since the mapping $T$ defined in Example \ref{example local not derivation} is defined on a complex JB$^*$-triple but it is not complex linear, it follows from the above proposition that $T$ is not triple derivation. 

After dealing with generalized real Cartan factors of rank one, we shall show that local triple derivations on generalized real Cartan factors of rank $>1$ are triple derivations.

\begin{proposition}\label{p Cartan factors of rank bigger than one} Let $C$ be a generalized real Cartan factor of rank $>1$ and let $T: C \to C$ be a linear map. The following are equivalent:\begin{enumerate}[$(a)$]
\item $T$ is a triple derivation;
\item $T$ is a local triple derivation;
\item $T$ is a triple derivation of the symmetrized triple product.
\end{enumerate}
\end{proposition}

\begin{proof} The implication $(b)\Rightarrow (c)$ is a consequence of \cite[Corollary 2.5]{BurFerPe2013}, while $(a)\Rightarrow (b)$ is clear. To prove $(c)\Rightarrow (a)$, let $T: C\to C$ be a triple derivation of the symmetrized triple product $<. \ ,.\ ,.>$. It is well known in real JB$^*$-triple theory that $\{\exp (t T) : C \to C\}_{t\in \mathbb{R}}$ is a one-parameter group of automorphisms of the symmetrized triple product. By \cite[Theorem 4.8 ]{IsKaRo95}, $\exp (t T)$ is a surjective isometry for every real $t$. Since $C$ is of rank $>1$, Corollary 2.15 in \cite{FerMarPe} assures that $\exp(t T)$ is a triple automorphism of the original triple product and hence $$\exp(t T) \J xyz = \J {\exp(t T) (x)}{\exp(t T) (y)}{\exp(t T) (z)},$$ for every $x,y,z\in C$ and $t\in \mathbb{R}.$ Finally, the identity $$\frac{\partial}{\partial t}{\mid_{_{t=0}}} \left(\exp(t T) \J xyz \right)=\frac{\partial}{\partial t}{\mid_{_{t=0}}} \left( \J {\exp(t T) (x)}{\exp(t T) (y)}{\exp(t T) (z)}\right),$$ gives $T\J xyz = \J {T(x)}yz + \J x{T(y)}z + \J xy{T(z)}.$
\end{proof}

Combining Propositions \ref{p local derivations in rank-one cartan factors} and \ref{p Cartan factors of rank bigger than one} we can also deduce that local triple derivations and triple derivations of the symmetrized triple product on a generalized real Cartan factor define the same objects.

\section{Local triple derivations on general real JB$^*$-triples}

In this section we shall find sufficient conditions on a general real JB$^*$-triple to guarantee that every local triple derivation on it is a triple derivation. We begin stating several properties of local triple derivations.\smallskip

We recall that a subspace $I$ of a real JB$^*$-triple $E$ is a
\emph{triple ideal} 
if $\{E,E,I\}+\{E,I,E\} \subseteq I$. 
It is known that a subtriple $I$ of $E$ is a triple ideal if and only if $\J EEI \subseteq I$ or $\J EIE \subseteq I$ or $\J EII\subseteq I$ (compare \cite{Bun86}).

\begin{lemma}\label{l local triple derivations and ideals} Let $T: E\to E$ be a triple derivation of the symmetrized triple product on a real JB$^*$-triple and let $I$ be a triple ideal of $E$. Then  $T (I) \subseteq I$.
\end{lemma}

\begin{proof} Given $a\in I,$ we can find $b$ in the JB$^*$-subtriple of $E$ generated by $a$ such that $\J bbb =a$ (notice that, in particular, $b$ lies in $I$). Since $T(a) = T\J bbb = 2 \J {T(b)}bb +  \J b{T(b)}b \in I,$ we get the desired statement.
\end{proof}

Since the bidual, $E^{**}$, of a real JB$^*$-triple $E$ is a JBW$^*$-triple (see \cite{IsKaRo95}), the following lemma, which is a variant of \cite[Proposition 2.1]{HoPerRu}, follows from the separate weak$^*$-continuity of the triple product of $E^{**}$ (cf. \cite{MarPe}) and the weak$^*$-density of $E$ in $E^{**}$.

\begin{lemma}\label{p bidual extensions}
Let $E$ be a real JB$^*$-triple and let $T : E \to E$ be a triple derivation of the symmetrized triple product. Then $T^{**} : E^{**} \to E^{**}$
is a weak$^*$-continuous triple derivation of the symmetrized triple product.$\hfill\Box$
\end{lemma}

The atomic decomposition established in \cite[Theorem 3.6]{PeSta} assures that every JBW$^*$-triple $W$ decomposes as an orthogonal sum
$$W = A \oplus^{\infty} N,$$ where $A$ and $N$ are
weak$^*$-closed ideals of $W$, $A$ being the weak$^*$-closed real linear
span of all minimal tripotents in $W$, $N$ containing no
\hyphenation{mi-nimal} minimal tripotents and $A \perp N$. It is also proved in \cite[Theorem 3.6]{PeSta} that $A$ is an
 orthogonal sum of weak$^*$-closed triple ideals which are generalized real Cartan factors.\smallskip

The tools required in our arguments will need the following Gelfand-Naimark type theorem for real
JB$^*$-triples, which has been borrowed from \cite[Proposition 3.1]{FerMarPe}.

\begin{proposition}\label{GN}
Let $E$ be a real JB*-triple. 
Let $A$ denote the atomic part of $E^{**},$ $j : E \hookrightarrow E^{**}$ the canonical embedding, and $\pi
: E^{**} \to A$ the canonical projection of $E^{**}$ onto $A$. Then $A$ writes as an
 orthogonal sum of weak$^*$-closed triple ideals which are generalized real Cartan factors, and the
mapping $\pi \circ j : E \to A$ is an isometric triple embedding.$\hfill\Box$
\end{proposition}

We are now in a position to give sufficient conditions on a real JB$^*$-triple $E$ to guarantee that every local triple derivation on $E$ is a triple derivation. We recall first, that, by \cite[Proposition 5.4]{Ka97}, every real JB$^*$-triple of rank one is precisely one of the
following rank-one generalized real Cartan factors: a rank-one type $I^{\mathbb{R}}$, $I^{\mathbb{C}}$, $I^{\mathbb{H}}$, a rank-one real spin factor $IV^{n,0}$, and $V^{\mathbb{O}}:= M_{1,2} (\mathbb{O})$.

\begin{theorem}\label{t sufficient conds on general real JB*-triples} Let $E$ be a real JB$^*$-triple whose second dual doesn't contain rank-one generalized real Cartan factors of types $I^{\mathbb{C}}$, $I^{\mathbb{H}}$ nor $V^{\mathbb{O}}:= M_{1,2} (\mathbb{O})$. Then every local triple derivation on $E$ is a triple derivation.
\end{theorem}

\begin{proof}
Let $T : E \to E$ be a local triple derivation on $E$. We have already commented that Corollary 2.5 in \cite{BurFerPe2013} assures that $T$ is a triple derivation of the symmetrized triple product. Lemma \ref{p bidual extensions} implies that $T^{**} : E^{**} \to E^{**}$ is a triple derivation of the symmetrized triple product. We have also argued that, by the atomic decomposition of $E^{**}$, there exists a family of mutually orthogonal, weak$^*$-closed triple ideals $\{C_i: i\in \Lambda\}\cup \{N\}$ of $E^{**}$ such that $$E^{**} =\left(\bigoplus_{i}^{\ell_{\infty}} C_i\right) \bigoplus^{\ell_{\infty}} N.$$ Lemma \ref{l local triple derivations and ideals} shows that $T^{**} (N) \subseteq N$ and $T^{**} (C_i) \subseteq C_i$, for every $i\in \Lambda$. By hypothesis, each $C_i$ is a generalized real Cartan factor of rank $>2$ or a rank-one generalized real Cartan factor of type $I^{\mathbb{R}}$ (i.e. $L(H,\mathbb{R})$, for a real Hilbert space $H$), or a real spin factor of rank one. Now, Lemmas \ref{l realtypeI} and \ref{l caseofspinfactor} and Proposition \ref{p Cartan factors of rank bigger than one} imply that $T^{**}|_{C_i} : C_i \to C_i$ is a triple derivation for every $i$, and hence $T^{**}|_{\bigoplus_{i}^{\ell_{\infty}} C_i} : \bigoplus_{i}^{\ell_{\infty}} C_i \to  \bigoplus_{i}^{\ell_{\infty}} C_i$ also is a triple derivation.\smallskip

Let us denote $A= \bigoplus_{i}^{\ell_{\infty}} C_i$ and keep the notation employed in Proposition \ref{GN}. Since the mapping $\Phi=\pi \circ j : E \to A$ is an isometric triple embedding, $T^{**}|_{A} : A \to A$ is a triple derivation, and $\Phi T = T^{**} \Phi$, we have $$\Phi T (\J xyz) =  T^{**}|_{A} \J {\Phi(x)}{\Phi(y)}{\Phi(z)} $$ $$= \J {T^{**}|_{A}\Phi(x)}{\Phi(y)}{\Phi(z)} + \J {\Phi(x)}{T^{**}|_{A}\Phi(y)}{\Phi(z)} + \J {\Phi(x)}{\Phi(y)}{T^{**}|_{A}\Phi(z)}$$
$$= \Phi \left(\J {T(x)}{y}{z} +  \J {x}{T(y)}{z} + \J {x}{y}{T(z)}\right),$$ witnessing that $T$ is a triple derivation.
\end{proof}

\begin{remark}\label{r Kaup rank one} The rank-one, generalized real Cartan factors avoided in Theorem \ref{t sufficient conds on general real JB*-triples} appear also in the problem of determining those real JB$^*$-triples $E$ satisfying that every surjective isometry on $E$ is a triple isomorphism. More precisely, W. Kaup established, in \cite[Lemma 5.12]{Ka97}, that a real JB$^*$-triple of rank one, $E$, satisfies that every surjective linear isometry on $E$ is triple isomorphism if and only if $E$ is isomorphic to a rank-one real Cartan factor of type $I^{\mathbb{R}}$ or a real spin factor of rank one, equivalently, $E$ is not isomorphic to one of the rank-one generalized real Cartan factors we have avoided in the statement of Theorem \ref{t sufficient conds on general real JB*-triples}.
\end{remark}

In order to get some applications of our main result, we refresh some definitions. A real JB$^*$-algebra is a norm closed self-adjoint real Jordan subalgebra of a (complex) JB$^*$-algebra. The class of J$^*$B-algebras introduced by K. Alvermann in \cite{Alv} coincides with the class of unital real JB$^*$-algebras. As in the complex setting, every real C$^*$-algebra is a real JB$^*$-algebra with respect to the Jordan product and every real JB$^*$-algebra is a real JB$^*$-triple with product $$\J abc = (a \circ b^*) \circ c + (c\circ b^*) \circ a - (a\circ c) \circ b^*$$ (see \cite{HanStor} and \cite{Li2003} for the basic background on JB- and JB$^*$-algebras and real C$^*$-algebras, respectively).\smallskip

The proof of Corollary 4.4 in \cite{FerMarPe} shows that the bidual of a real JB$^*$-algebra can only contain rank-one generalized real Cartan factors of the form $\mathbb{R}$, $\mathbb{C}$ or $IV^{n,0}$. So, the following result follows a direct consequence of Theorem \ref{t sufficient conds on general real JB*-triples}.

\begin{corollary}
Every local triple derivation on a real JB$^*$-algebra or a real C$^*$-algebra is a triple derivation.$\hfill\Box$
\end{corollary}

The counter-example exhibited in Example \ref{example local not derivation} and the consequent restrictions appearing in the real setting are inviting us to explore the connections with 2-local triple derivations. In the associative setting, 2-local derivations and 2-local homomorphisms were introduced by P. \v{S}emrl in \cite{Semr}. Given a associative algebra $A$, a mapping $T : A \to A$ is called a \emph{2-local derivation} if for every $a, b\in  A$ there is an (associative) derivation $D_{a,b} : A\to A$, depending on $a$ and $b$, such that $D_{a,b}(a) = T(a)$ and $D_{a,b}(b) = T(b)$. In the just quoted paper, \v{S}emrl shows that if $H$ is an infinite-dimensional separable Hilbert space, then every 2-local derivation on $L(H)$ is a derivation of $L(H)$. \v{S}emrl's paper has motivated over forty publications during the last years. Here we extend the notion to the triple setting. Let $T: E \to E$ be a mapping on a real or complex JB$^*$-triple. We shall say that $T$ is a \emph{2-local triple derivation} when given elements $a, b\in  A$ there is a triple derivation $\delta_{a,b} : E\to E$, depending on $a$ and $b$, such that $\delta_{a,b}(a) = T(a)$ and $\delta_{a,b}(b) = T(b)$.\smallskip

Every 2-local triple derivation $T: E\to E$ is $1$-homogeneous, that is, $T(\lambda a) = \lambda T(a)$, for every $a$ in $E$ and every scalar $\lambda$. Indeed, let us take a triple derivation $\delta_{a, \lambda a}: E\to E$ such that $T(\lambda a) = \delta_{a, \lambda a} (\lambda a) = \lambda \delta_{a, \lambda a} (a) = \lambda T(a)$. The above Example \ref{example local not derivation} provides an example of a local triple derivation on a rank-one real JB$^*$-triple which is not a triple derivation. This kind of counterexample is not valid for additive 2-local triple derivation. Indeed, let $T : E \to E$ be an additive 2-local triple derivation on a real JB$^*$-triple. We have already seen that, in this case, $T$ is linear. For each derivation $\delta : E \to E$ the mapping $\widetilde{\delta}: \widetilde{E} \to \widetilde{E}$, $\widetilde{\delta} (x+i y):= \delta(x) + i \delta (y)$ defines a $\mathbb{C}$-linear triple derivation on the complexification of $E$ (compare \cite[Remark 1]{HoMarPeRu}). We also consider a $\mathbb{C}$-linear mapping on $\widetilde{E}$ defined by $\widetilde{T} (x+i y):= T(x) + i T (y)$. Let us take $a,b\in E$ and a triple derivation $\delta_{a,b}: E \to E$ satisfying $\delta_{a,b}(a) = T(a)$ and $\delta_{a,b}(b) = T(b)$. In this case, the mapping $\widetilde{\delta}_{a,b}$ is a derivation on $\widetilde{E}$ and $\widetilde{T} (a+i b) = T(a) + i T(b) = \delta_{a,b}(a)+ i \delta_{a,b}(b) = \widetilde{\delta}_{a,b} (a+i b)$. Therefore, $\widetilde{T}$ is a local triple derivation on the complex JB$^*$-triple $\widetilde{E}$, and hence Theorems 2.4 and 2.8 in \cite{BurFerPe2013} assure that $\widetilde{T}$ (and hence $T$) is a (continuous) triple derivation.

\begin{corollary}\label{cor 2-local} Let $T: E \to E$ be a (not necessarily continuous) additive 2-local triple derivation on a real JB$^*$-triple. Then $T$ is a continuous triple derivation. $\hfill\Box$
\end{corollary}

\begin{problem}\label{problem 2-local triple derivations on real JB*-triples} Is every 2-local triple derivation on a real JB$^*$-triple additive? Equivalently, is every 2-local triple derivation on a real JB$^*$-triple a triple derivation?
\end{problem}

\bigskip

\end{document}